\documentclass{amsart}
\usepackage[T1]{fontenc}
\usepackage{lmodern}
\usepackage{textcomp}

\usepackage{amssymb, amsthm, amsmath}
\usepackage{mathtools} 
\usepackage{float}
\usepackage{verbatim}
\usepackage{dsfont}

\usepackage{enumitem}
\usepackage{tikz}
\usepackage{graphicx}
\usepackage{commath}
\usepackage{hyperref}
\usepackage[capitalize]{cleveref}
\hypersetup{pdfpagemode={UseOutlines},
	bookmarksopen=true,
	bookmarksopenlevel=0,
	hypertexnames=false,
	colorlinks=true, 
	citecolor=teal, 
	linkcolor=blue,
	urlcolor=magenta,
	pdfstartview={FitV},
	unicode,
	breaklinks=true,
}
\makeatletter
\@namedef{subjclassname@2020}{\textup{2020} Mathematics Subject Classification}
\makeatother

\newtheorem{lemma}{Lemma}[section]
\newtheorem{corollary}[lemma]{Corollary}
\newtheorem{proposition}[lemma]{Proposition}
\newtheorem{example}{Example}[section]
\newtheorem{claim}{Claim}
\newtheorem{theorem}[lemma]{Theorem}

\newtheorem{subclaim}{Subclaim}

\newtheorem{thmintro}{Theorem}

\theoremstyle{definition}
\newtheorem{definition}[lemma]{Definition}

\newcommand{\p}{\mathbb{P}}
\newcommand{\e}{\mathbb{E}}

\newcommand{\N}{\mathbb{N}}

\theoremstyle{plain}
\newcommand{\Fill}{\text{Fill}}

\newcommand{\calp}{\mathcal{P}}
\newcommand{\cay}[1]{\operatorname{Cay}(#1)}
\newcommand{\free}[1]{F(#1)} 
\newcommand{\area}[2]{\operatorname{Area}_{#1}(#2)}

\title{Random Dehn function of groups}

\author[Jerónimo García-Mejía]{Jerónimo García-Mejía}
    \address{(Jerónimo García-Mejía) Mathematical Institute, Andrew Wiles Building, Observatory Quarter, University of Oxford, Oxford OX2 6GG, UK}
    \email{jeronimo.garcia-mejia@maths.ox.ac.uk}
    
\author[Antoine Goldsborough]{Antoine Goldsborough}
    \address{(Antoine Goldsborough) Department of Mathematics, Tata Institute of Fundamental Research, Mumbai, India}
    \email{antoinegold10@gmail.com}

\keywords{Dehn functions, filling invariants, random walks, Markov chains, acylindrically hyperbolic}

\subjclass[2020]{Primary: 20P05. Secondary: 20F65, 20F69,20F67}
    
\begin{document}
\begin{abstract}
	In this note, we study the notion of random Dehn function and compute an asymptotic upper bound for finitely presented acylindrically hyperbolic groups whose Dehn function is at most polynomial. By showing that in these cases, if the group is not hyperbolic, then the random Dehn function is strictly smaller than the usual Dehn function we confirm Gromov's intuition albeit in a different model. In fact, we show that in these cases the random Dehn function is at most quadratic.
\end{abstract}

\maketitle

\section{Introduction}

Gromov's seminal program of classifying groups up to quasi-isometries has been a central problem in geometric group theory. Among the quasi-isometry invariants that have emerged, the \emph{Dehn function} of a finitely presented group has played a central role. The Dehn function of a finitely generated group $G$ is a complexity measure of a direct attack of the \emph{word problem} in $G$: it provides an upper bound on the \emph{area of a word} of length at most $n$, that is, on the number of relations that must be applied to a word in the generators of $G$, in order to prove that it represents the trivial element in $G$; the bound is expressed as a function of the length of the word. Geometrically, it provides an upper bound on the area of filling discs for a loop of length at most $n$ in a simply connected Riemannian manifold on which $G$ acts geometrically; the bound is given in terms of the length of the loop. An outstanding feature of Dehn functions is that they characterise hyperbolic groups as precisely those finitely presented groups having subquadratic Dehn function  \cite{gromov1992asymptotic,OlshanskiHyperbolic,BowditchSubquadratic}. 

The Dehn function of a finitely presented group captures the worst-case complexity of the word problem in the group. However, for many algorithmic problems it is natural to inquire about the \emph{average-case} complexity. Gromov defined \cite[5.A6.c]{gromov1992asymptotic} the \emph{averaged Dehn function} of a group $G$ as the average of the area over all words of length at most $n$ that represents the identity in $G$, and claimed that in several cases the asymptotic growth of the averaged Dehn function should be strictly smaller than that of the Dehn function. Gromov's claim was first verified for finitely generated abelian groups by Bogopolski and Ventura \cite{BogopolskiVentura}, and extended to all finitely generated nilpotent groups by Young \cite{young2008averaged}.

Using another model of filling function, which we call \emph{random Dehn function}, see \cref{def:randomdehn}, we confirm Gromov's intuition by showing that for finitely presented (non-hyperbolic) acylindrically hyperbolic groups whose usual Dehn function is at most polynomial, the random Dehn function grows strictly slower than the usual Dehn function. Roughly speaking, the random Dehn function of a group $G$ is defined as the expected values of areas in $G$. Intuitively, this is done by considering a random walk $(Z_n)_n$ on $G$, which we think as a path on $G$. We then join the extreme points of this path, using a (quasi-geodesic) combing $\alpha$ for $G$, to obtain a closed loop on $G$. The random Dehn function $R\delta_{(G, (Z_n), \alpha)}$ of $G$ is then defined by taking the expected values of areas of loops constructed in this way.

Our main results, confirming Gromov's intuition, are the following, with the usual order on functions which we recall later.

\begin{thmintro}\label{thmintro:ah_rdehn}
    Let $G$ be a finitely presented acylindrically hyperbolic group such that its Dehn function $\delta_G$ is at most a polynomial of degree $d \geq 1$. Let $(Z_n)_n$ be a simple random walk on $G$. For every quasi-geodesic combing $\alpha$ of $G$, we have that
    \[
        R\delta_{(G,(Z_n), \alpha)}(n)  \preccurlyeq  \frac{n}{\log (n)}\delta_G(\log(n)) \preccurlyeq n^2.
    \]
    In particular, if $G$ is hyperbolic, then $R\delta_{(G,(Z_n), \alpha)}(n) \preccurlyeq n$. Else, if $d=2$ then $R\delta_{(G,(Z_n), \alpha)}(n) \preccurlyeq n \log n \prec \delta_G (n)$, and if $d>2$, then $R\delta_{(G,(Z_n), \alpha)} \prec \delta_G$.
\end{thmintro}

We note that such a model of random Dehn function has been considered by Sisto in \cite[Theorem 4.10]{SistoTrackingRates} in the case of relatively hyperbolic groups. Sisto shows that the random Dehn function of a  hyperbolic group, relative to proper subgroups with at most polynomial Dehn function, is $\preccurlyeq n\delta(\log(n))$ where $\delta$ is the maximum of the Dehn functions of the peripheral subgroups. 

Random walks belong to the larger class of stochastic processes called \emph{tame quasi-homogeneous Markov chains}, whose definition we recall in \cref{sec:preliminaries}. These processes are invariant under quasi-isometries. In the case of tame Markov chains, we obtain the following result.

\begin{thmintro}\label{thmintro:tame_MC}
Let $G$ be one of the following:
\begin{itemize}
    \item A non-elementary hyperbolic group;
    \item a non-elementary relatively hyperbolic group; 
    \item an acylindrically hyperbolic $3$-manifold group;
    \item the fundamental group of a graph of groups with nilpotent undistorted edge groups and at most polynomial Dehn function
\end{itemize}
Let $(w_n^\ast)$ be a tame Markov chain on $G$ and $\alpha$ any quasi-geodesic combing of $G$, then we have that 
	\[
	R\delta_{(G,(w^*_n), \alpha)}(n) \preccurlyeq  \frac{n}{\log (n)}\delta_G(\log(n)) \preccurlyeq n^2. 
	\]
In particular, if $G$ is not hyperbolic, then $R\delta_{(G,(w^*_n), \alpha)}(n) \preccurlyeq n \log n \prec \delta_G(n)$.
\end{thmintro}

A subfamily of tame Markov chains, which is still invariant by quasi-isometry and includes simple random walks are \emph{tame quasi-homogeneous Markov chains}. If we consider these, then we obtain the following result.

\begin{thmintro}\label{thmintro:cor}
Let $G$ be one of the following:
\begin{itemize}
    \item the mapping class group of a finite type surface;
    \item an extra-large Artin group; 
    \item or more generally a well-behaved hierarchically hyperbolic group (in the sense of \cite{InducedQIs}).
\end{itemize}
Let $(w_n^\ast)$ be a tame quasi-homogeneous Markov chain on $G$. Then, for any quasi-geodesic combing $\alpha$ of $G$ we have that:
	\[
		R\delta_{(G,(w^*_n), \alpha)}(n) \preccurlyeq  \frac{n}{\log(n)} \delta_G(\log (n)) \preccurlyeq n^2.
	\]
In particular, if $G$ is not hyperbolic, then $R\delta_{(G,(w^*_n), \alpha)}(n) \preccurlyeq n \log n \prec \delta_G(n)$.
\end{thmintro}

The proof of both of these theorems relies on \emph{deviation inequalities} for the appropriate stochastic processes. Roughly, these inequalities state that the probability that the distance between the position of a random walk (or tame quasi-homogeneous Markov chain) after $k$ steps and a quasi-geodesic between the starting point and the position after $n$ steps is at least $t$ decays exponentially in $t$, see Theorem \ref{thm:small_devns} for the exact statement. We then use these inequalities to deduce that our sample path coming from the stochastic process, with high probability, stays close to a quasi-geodesic which is enough to bound the area of the filling and hence give an estimate for the random Dehn function.

It is worth noting that another quasi-isometry invariant based on tame Markov chains has been considered by the second named author and A. Sisto in \cite{rdiv1, rdiv2}. Instead of looking at the Dehn function of generic points, there the authors look at the divergence of generic points.

\subsection*{Acknowledgments}
The authors thank the organisers of the thematic program on geometric group theory held at the Montreal Centre de Recherches Scientifiques in Spring 2023, where part of this work took place. The authors also thank Kunal Chawla and Alessandro Sisto for interesting discussions. The first author gratefully acknowledge funding by the DFG 281869850 (RTG 2229) and the Royal Society of Great Britain. The second author was supported by the EPSRC-UKRI studentship EP/V520044/1 and the LMS Early Career Fellowship. The authors are grateful to the referee for their valuable feedback and comments. Finally, we are thankful to Robert Young for valuable comments and for pointing out an imprecision on a previous version; correcting it lead to a strengthening of our results.

\section{Preliminaries}\label{sec:preliminaries}
In this section, we recall some notions from geometric group theory as well as from the theory of random walks and Markov chains. We then define the notion of random Dehn function of a group. In \cref{prop:randomDehn_QIinvariant} we show that in the setting of non-amenable groups, by considering the class of tame Markov chains, the random Dehn function is a quasi-isometric invariant.

\subsection{Acylindrically hyperbolic groups} 
We recall the definition of an acylindrically hyperbolic groups, which is a generalisation of the notion of a hyperbolic group, and we refer to \cite{Osin} and the references therein for more properties and examples of this class of groups.

\begin{definition}
    We say that a group acts \textit{acylindrically} on a metric space $X$ if for every $R \geq 0$, there exist $N > 0$, $L > 0$ such that for every $x,y \in X$ with $d_X(x,y) \geq L$, we have \[ |\{g \in G \ | \ d_X(x,gx) \leq R, \ d_X(y,gy) \leq R \}| \leq N. \]
    
    A group is \textit{acylindrically hyperbolic} if it is non-virtually cyclic and acts acylindrically on a hyperbolic geodesic metric space with unbounded orbits. 
\end{definition}

\begin{example}[\cite{Osin}] The following groups are acylindrically hyperbolic
    \begin{itemize}
        \item Non-elementary hyperbolic groups;
        \item Non virtually cyclic relatively hyperbolic groups with proper peripheral subgroups;
        \item All but a finite number of mapping class groups of connected oriented surfaces;
        \item $Out(F_n)$ for $n \geq 2$;
        \item Non virtually cyclic groups acting properly on a CAT(0) space with a rank-1 element.
    \end{itemize}
\end{example}

\subsection{Markov chains}

We usually denote a Markov chain on a group by $(w_n^\ast)_n$, where $w_n^p$ denotes the position of the Markov chain starting at $p$ after $n$ steps and the quantity $\mathbb P [w_n^o = g]$ denotes the probability of getting from $o$ to $g$ in $n$ steps. We recall some definitions from \cite{GS21} and \cite{InducedQIs} and refer to these for more details.

\begin{definition}[Tame Markov chain]\label{def:tame-markov-chain}
    For a finitely generated group $G$ with finite generating set $S$, with word metric $d_S$, we say that a Markov chain $(w_n^\ast)$ on $G$ is \emph{tame} if it satisfies the following conditions:
    \begin{enumerate}
        \item \emph{Bounded jumps}: There exists a $K>0$ such that for all $g \in G, n \in \mathbb N$ we have that $d_S(w_n^p, w^p_{n+1}) \leq K$,
        \item \emph{Non-amenability}: there exists $A>0$ and $\rho <1$ such that for all $g,h \in G$ and $n \geq 0$ we have $\mathbb P [w_n^g = h] \leq A\rho^n$,
        \item \emph{Irreducibility}: For each $s \in G$ there exists constants $\epsilon_s, K_s >0$ such that for all $g \in G$ we have $\mathbb P [w_k^g = gs] \geq \varepsilon_s$ for some $k \leq K_s$.
    \end{enumerate}
\end{definition}

\begin{definition}
\label{defn:push-forward}
    For a bijection $f: G \rightarrow H$ and a Markov chain $(w^\ast_n)$ on $G$, there is a natural Markov chain in $H$ induced by $f$ called the \emph{push-forward Markov chain} on $H$, which is denoted by $f_\#(w^p_n)$. This is the Markov chain such that $\mathbb P [f_\#(w^o_n)=h] = \mathbb P [w_n^o = f^{-1}(h)]$ for all $h \in H$, $o \in G$ and $n \geq 0$.
\end{definition}

Another property of Markov chains that we consider is the following. 

\begin{definition}
    A Markov chain $(w^p_n)_n$ in $G$ is \emph{quasi-homogeneous} if for some $D > 0$ it has the following property. For every $g,h \in G$ there is a bijection $f: G \rightarrow G$ such that $f(g) = h$ and for all $n \in \mathbb N$ and all $q \in G$ we have that $f_\#(w^q_n) = (w_n^{f(q)})$.
\end{definition}

The following result shows that appropriate random walks, including simple random walks, are tame quasi-homogeneous Markov chains.
\begin{lemma}[{\cite[Lemma 2.8]{GS21} \cite{InducedQIs}}]
    \label{lem:random_walks_are_tame}
    Let $G$ be a non-amenable group and let $(z_n)_n$ be a random walk on $G$ driven by $\mu$ such that the support of $\mu$ is finite and generates $G$ as a semigroup. Then $(z_n)_n$ is a tame quasi-homogeneous Markov chain.
\end{lemma}

\subsection{Dehn and random Dehn functions}

\begin{definition}
    Let $S$ be a non-empty set and $\free S$ the free group on $S$. Let $G$ be a group and suppose it is given by a finite presentation $\calp = \langle S|R \rangle$. We denote by $|\cdot|_S$ the \emph{word-length} with respect to $S$ and by $d_S$ the induced metric. We say that a word $w$ in $S$ is \emph{null-homotopic} in $G$ if it represents the identity in $G$ and write $w =_\calp 1$. We define the \emph{area} of a null-homotopic word $w$ in $G$ with respect to $\calp$ as the positive integer defined by
    \[
    \area{\calp}{w} := \min \{ \ell \in \mathbb{N} \mid w =_\calp \prod_{i = 1}^{\ell} a _{i} r _{i} ^{\epsilon _{i}} a _{i} ^{-1}, \ \text{with} \ a_i \in \free{S}, \ r_i \in R, \ \text{and} \ \epsilon_i \in \{\pm 1\} \}.
    \]
    The \emph{Dehn function of $\calp$} is the function $\delta_\calp: \N \rightarrow \N$ defined as 
    \[
    \delta_\calp (n) := \max\{\area{\calp}{w} \mid w =_\calp 1, \ |w|_S \leq n \}.
    \]
\end{definition}

Changing the presentation of $G$ does not alter the asymptotic behaviour of the Dehn function of the presentations. Concretely, the \emph{Dehn function of $G$} is well-defined up to $\asymp$- equivalence of functions \footnote{Recall, two functions $f,g: [0, \infty) \rightarrow [0, \infty)$ are said to be \emph{$\asymp$-equivalent} if $f \preccurlyeq g$ and $g \preccurlyeq f$, where $f \preccurlyeq g$ means that there exists constants $A,B >0$ and $C,D,E \geq 0$ such that $f(n) \leq Ag(Bn +C) + Dn +E$ for all $n \geq 0$.}, so we write $\delta_G = \delta_\calp$. In fact, the Dehn function of a group is a well-defined quasi-isometry invariant up $\asymp$-equivalence (for details see \cite{AlonsoQI} and \cite{BridsonHaefliger}).

For a group $G$ generated by $S$, we can assume without loss of generality that $S$ is symmetric, namely that $S=S^{-1}$. There is a natural surjective map, called \emph{evaluation map}, from the free monoid generated by $S$, denoted by $S^\ast$, to $G$ defined by sending each element in $S^\ast$, a \emph{word in $S$}, into the element it represents in $G$.
\begin{definition}
    A morphism of monoids $\alpha: G \rightarrow S^\ast$ is called a \emph{combing for $G$} if it is a (set-theoretic) right inverse of the evaluation map. Geometrically, given $g \in G$ we can identify $\alpha(g)$ with a continuous path $\alpha_g \colon [0,l] \rightarrow \cay{G,S}$ connecting $1$ to $g$ which we call the \emph{combing line of $g$}.
\end{definition} 

In other words, a combing of $G$ is a choice of normal form for elements in $G$.

\begin{definition}
    Let $D \geq 1$. We say that a combing $\alpha \colon G \rightarrow S^\ast$ is a \textit{$D$-quasi-geodesic combing} if for every $g \in G$ the combing line $\alpha_g$ is a $(D,D)$-quasi-geodesic. We simply say that $\alpha$ is a \emph{quasi-geodesic combing} if there is some $D \geq 1$ such that $\alpha$ is a $D$-quasi-geodesic combing.
\end{definition}

\begin{definition}\label{def:fill}
    Let $G$ be a finitely presented groups with finite presentation $\calp = \langle S|R \rangle$ and $\alpha: G \rightarrow S^\ast$ be a combing for $G$. For $x_0, \ldots, x_n \in G$ let \[\alpha_i = \alpha(x_{i+1}x_i^{-1}) \quad \text{for} \  0 \leq i \leq n-1, \ \text{and} \quad \alpha_n = \alpha(x_n).\] For the null-homotopic word in $S$ defined as  $W_{x_1, \ldots, x_n} = \alpha_1 \cdot \ldots \cdot \alpha_{n-1} \cdot \alpha_n^{-1}$
    we denote 
    \[
    \Fill_{\calp, \alpha}(x_1, \ldots, x_n) = \area{\calp}{W_{x_1, \ldots, x_n}}.
    \]
\end{definition}

\begin{definition}[Random Dehn function]\label{def:randomdehn}
    Let $G$ be a finitely presented group with finite presentation $\calp = \langle S|R \rangle$ and let $\alpha: G \rightarrow S^\ast$ be a quasi-geodesic combing for $G$. Let be $(w_n^\ast)$ a stochastic process on $G$. For every starting point $o \in G$ of the stochastic process set 
    \[
    x_{o,1} = w_0^o, \ldots, x_{o,n} = w_n^o, \quad \alpha_{o,i} = \alpha(x_{o,i+1} x_{o,i}^{-1}) \quad \text{for} \ 0 \leq i \leq n-1, \ \text{and} \quad \alpha_{o,n} = \alpha(x_{o,n}).
    \]
    The \emph{random Dehn function of $\calp$ with respect to $(w_n^\ast)$ and $\alpha$} is the function $R\delta_{(\calp,(w_n^\ast), \alpha)} \colon \N \rightarrow \N$ defined as
    \[
    R\delta_{(\calp,(w_n^\ast), \alpha)}(n) = \sup_{o \in G} \mathbb E [ \Fill_{\calp, \alpha}(x_{o,1}, \ldots, x_{o,n})].
    \]
\end{definition}
The stochastic process $ (w^\ast_{n} )_n$ that we consider is either a random walk driven by an admissible measure, or a tame Markov chain. The reason for looking at tame quasi-homogeneous Markov chains instead of just random walks is that together with quasi-geodesic combings these make the notion of random Dehn function invariant under quasi-isometries, as we now show.

\begin{proposition}\label{prop:randomDehn_QIinvariant} Let $G$ and $H$ be two finitely generated non-amenable groups. Suppose that $H$ is quasi-isometric to $G$. Then for all tame Markov chains $(z_n^{*})_n$ on $G$ and quasi-geodesic combing $\alpha$ on $G$, there exists a tame Markov chain $(w_n^{*})_n$ on $H$ and quasi-geodesic combing $\beta $ on $H$ such that for all $p \in G$
    \[ 
    R\delta_{(G,(z_n^\ast),\alpha)}(n) \asymp R\delta_{(H,(w_n^\ast),\beta)}(n). 
    \]
\end{proposition}
\begin{proof}
   By \cite[Theorem 4.1]{Whyte}, any quasi-isometry between finitely generated non-amenable groups is at bounded distance from a bijective quasi-isometry $f\colon G \to H$. We let $w_n^* = f_{\#}(z_n^*)$ be the push-forward of $z_n^*$ by $f$ as defined in \cref{defn:push-forward}, then by \cite[Lemma 2.8]{GS21} and \cite[Remark 2.9]{InducedQIs} this new stochastic process is a tame quasi-homogeneous Markov chain on $H$. Let $\beta = f \circ \alpha$, i.e. for all $x,y \in H$ let $\beta(x,y) = f(\alpha(f^{-1}(x), f^{-1}(y))$. The rest of the proof follows directly from the fact that the usual Dehn function is a quasi-isometric invariant. 
\end{proof}

\section{Bounding the random Dehn function}

We first start with a result that shows that if a Markov chain makes linear progress then sub-walks of length at least logarithmic length also make linear progress. This is a generalisation of a result from \cite{SistoTrackingRates}. To state the result we recall the following definition.

\begin{definition}[Linear progress with exponential decay]
We say that a Markov chain $(w^{*}_n)$ on $G$ \emph{makes linear progress with exponential decay in $G$} if there exists a constant $K >0$ such that for all $o \in G$ and $k$  we have
\begin{equation}\label{eq:exponential-decay}
   		\mathbb P \left[d_S(o,w_k^{o}) < \frac{k}{K} \right] \leq K e^{-k/K}.
   \end{equation}
\end{definition}

\begin{lemma}[Generalisation of {\cite[Lemma 4.5]{SistoTrackingRates}}]
\label{lem:subwalks_linear_progress}
 Let $(w^{*}_n)$ be a Markov chain on $G$ that makes linear progress with exponential decay in $G$, see \eqref{eq:exponential-decay}. Then for each $r\geq 1$, there exists $C_3>0$ so that for each $n \geq 1$ and $o \in G$ we have:
    \[
    \mathbb P \left[ \exists i,j \leq n : \vert i-j\vert \geq C_3\log(n), \ d_X(w^{o}_i, w^{o}_j) \leq \frac{\vert i-j \vert}{C_3} \right] \leq C_3 n^{-r}.
    \]
\end{lemma}

\begin{proof}[Proof of \cref{lem:subwalks_linear_progress}]
   The proof is very similar to the proof in the case of a simple random walk proved in \cite[Lemma 4.5]{SistoTrackingRates}.
   Let $K$ be the constant from linear progress with exponential decay of the Markov chain in $G$, see \eqref{eq:exponential-decay}, and let $r \geq 1$ and let $C_3=(r+1)K$. We have, for all $o \in G$ and $i<j$:
   \[ 
   \mathbb P \left[d_X(w^{o}_i, w^{o}_j) \leq \frac{\vert j-i \vert}{K}  \right]= \sum_{g \in G}\mathbb P \left[d_X(g, w^g_{j-i}) \leq \frac{\vert j-i \vert}{K} \,\middle\vert\, w_i^o=g\right]\mathbb P \left[w_i^o=g\right] \leq Ke^{-\vert j-i \vert/K}, 
   \]

where we used the Markov property. Hence, by the choice of $C_3$ we get 
    \begin{align*}
            \mathbb P \left[ \exists i,j \leq n : \vert i-j\vert \geq C_3\log(n), \ d_X(w^{o}_i, w^{o}_j) \leq \frac{\vert i-j \vert}{C_3} \right] &\leq \sum_{i=0}^n \left( \sum_{l \geq C_3\log(n)}K e^{-l/K} \right) \\
            &\leq K n \cdot n^{-C_3/K} \sum_{t \geq 0} e^{-t/K} \\
            &\leq C_3n^{-r}. \qedhere
    \end{align*}
\end{proof}

The following result is a direct consequence of the deviation inequalities from \cite[Theorem 1.1]{MathieuSisto} and \cite[Theorem 1.4]{GS21} 

\begin{theorem}[Consequence of \cite{GS21, MathieuSisto}]
\label{thm:small_devns}

We consider one of the following set ups. 
\begin{itemize}
    \item Let $(w^*_{n})_{n}$ be a simple random walk on an acylindrically hyperbolic group $G$, 
    \item Let $(w^*_{n})_{n}$ be a tame Markov chain on a group satisfying \cite[Assumption 1 or Assumption 2]{GS21}, for example a hyperbolic group, or relatively hyperbolic group or the fundamental group of a graph of groups with nilpotent undistorted edge groups,
    \item Let $(w^*_{n})_{n}$ be a tame quasi-homogenenous Markov chain on a well-behaved hierarchically hyperbolic group, for example the mapping class group of a finite-type surface or an extra-large Artin group.
\end{itemize}
Then for all $D>0$, there exists a constant $C_1$ such that the following holds for all $n \in \mathbb N$
\[
\mathbb P \left[ \forall  k \leq n \colon \sup_{\alpha \in QG_{(D,D)}(1, Z_n)} d_S(Z_k, \alpha) \geq l  \right] \leq C_1ne^{-l/C_1}
\]
the supremum is taken over the set $QG_{(D,D)}(x,y)$ of all $(D,D)$-quasi-geodesics from $x$ to $y$ in $G$.
\end{theorem}
\begin{proof}
This follows immediately from the fact that in each of these set ups we have deviation inequalities by \cite[Theorem 1.1]{MathieuSisto} and \cite[Theorem 1.4]{GS21} (where for the last bullet point this is true via  \cite[Theorem 3]{InducedQIs}), hence summing up with a union bound over all $k \leq n $ gives the desired result.
\end{proof}

\begin{theorem}[]\label{thm:ah_rdehn}
Let $G$ be a finitely presented group with presentation $\calp = \langle S|R\rangle$ and with Dehn function $\delta_G$ at most polynomial. Let $(w_n^*)_n$ be a tame Markov chain on $G$ that satisfies the conclusion of Theorem \ref{thm:small_devns}. Then, for all quasi-geodesic combings $\alpha$ of $G$, we have that 
\[
R\delta_{(G,(w_n^*), \alpha)}(n) \preccurlyeq \frac{n}{\log(n)}\delta_G(\log(n)) \preccurlyeq n^2.
\]
\end{theorem}

\begin{proof}[Proof of \cref{thm:ah_rdehn}]

We show that for all $D>0$, there exists a constant $C>0$ such that for all $D$-quasi-geodesic combings $\alpha$ of $G$ we have that 
\[ 
R\delta_{(G,(w_n^*), \alpha)}(n) \leq  C\frac{n}{\log(n)}\delta_G(\log(n))+Cn+C. 
\]

Let $d$ be the degree of a polynomial bounding (from above) the Dehn function of $G$, the existence of $d$ is guaranteed by our assumptions. Let $(w_n^*)_n$ be one of the above stochastic processes on the corresponding group $G$, the important part is that these satisfy the conclusion of \cref{thm:small_devns}. Let $o \in G$, $n \in \mathbb N$, and $\alpha_{o, w^0_n}$ be the quasi-geodesic line from $o$ to $w^{o}_n$, which we assume to be a $D$-quasi-geodesic. 

We now provide a bound which is uniform over $o \in G$. Let $C_1$ be the constant from Theorem \ref{thm:small_devns} associated to $D$, which we increase to make sure that $n^dC_1n^{1-C_1} \to 0$. Let $\mathcal A_n$ be the event 
	\[ 
		``\forall k \leq n : d_S(\alpha_{o, w^0_n}, w_k^{o}) \leq C_1^2 \log(n) ''. 
	\] 
By Theorem \ref{thm:small_devns} the event $\mathcal A_n$ holds with probability $\mathbb P \left[\mathcal A_n \right] \geq 1-C_1n^{1-C_1} \to 1$.

We now work under the assumption that the event $\mathcal A_n$ holds, i.e. the event $\forall k \leq n : d_S(\alpha_{o, w^0_n}, w^{o}_k) \leq C_1^2\log(n)$. Hence, for all $k \leq n$, we can choose discrete geodesics $\gamma_k$ from $w^{o}_k$ to $\alpha_{o, w^0_n}$ of length at most $C^2_1\log(n)$. 

Let $p_k=\alpha_{o, w^0_n} \cap \gamma_k$. For two points $x,y \in \alpha_{o, w^0_n}$,  we write $\ell(\alpha_{o, w^0_n}\vert_{[x,y]})$ to denote the length, along $\alpha_{o, w^0_n}$, between $x$ and $y$.

Since $G$ is non-amenable, it follows that a tame Markov chain makes linear progress in $G$ (by the non-amenability condition of tameness). Therefore we can use Lemma \ref{lem:subwalks_linear_progress} to get that subwalks of logarithmic length also make linear progress in the Cayley graph of $G$ with respect to the chosen generating set. 

Let $C_3$ be as in Lemma \ref{lem:subwalks_linear_progress} for $r=d+1$ ($d$ being the polynomial degree bounding the Dehn function $\delta_G$ of $G$) and for all $i \geq 0$, define integers $k_i=\lceil 100iC_1^2C_3D^2\log(n) \rceil$ and let $t$ be such that $k_t \leq n$ but $k_{t+1} >n$.  Define $\mathcal B_n$ to be the complementary event to the defining event of the probability in Lemma \ref{lem:subwalks_linear_progress}, i.e. $\mathcal B_n$ is the event 
\[
\text{``$\forall  i,j \leq n : \vert i-j \vert < C_3 \log(n) \quad \text{or} \quad d_S(w^{o}_{i}, w^{o}_{j}) >  \frac{\vert i-j \vert}{C_3} $''}.
\]
By Lemma \ref{lem:subwalks_linear_progress}, we have that $\p \big[\mathcal B_n^C \big] \leq C_3n^{-k}$ and hence it makes sense to consider the case where $\mathcal B_n$ holds, as we now do in the following claim.

\begin{claim}
\label{claim:projections_increase}
   If the events $\mathcal A_n$ and $\mathcal B_n$ both hold, then for all $k_{i+1} <n $ we have $\ell(\alpha_{o, w^0_n}\vert_{[o,p_{k_{i+1}}]})-\ell(\alpha_{o, w^0_n}\vert_{[o,p_{k_{i}}]})>0$.
\end{claim}
\textit{Proof of Claim \ref{claim:projections_increase}.}
For a contradiction, assume that this is not the case. Hence, there exists $k_i$ such that $\ell(\alpha_{o, w^0_n}\vert_{[o,p_{k_{i+1}}]})-\ell(\alpha_{o, w^0_n}\vert_{[o,p_{k_{i}}]})\leq 0$. If this is the case, then 

\begin{subclaim}
\label{subclaim:later_on_close}
    There exists $j \geq k_{i+1}$ such that $\ell(\alpha_{o, w^0_n}\vert_{[p_j,p_{k_{i}}]}) \leq 10C_1^2D\log(n)$.
\end{subclaim}

\textit{Proof of \cref{subclaim:later_on_close}}
Let $q>k_{i+1}$ be the minimal index such that  $\ell(\alpha_{o, w^0_n}\vert_{[o,p_{q}]}) -\ell(\alpha_{o, w^0_n}\vert_{[o,p_{k_{i}}])}>0$ (which exists as $k_{i+1} <n$). If $\ell(\alpha_{o, w^0_n}\vert_{[p_q,p_{k_{i}}]}) \leq 10C_1^2\log(n)$ then $j=q$ proves the subclaim.
If $\ell(\alpha_{o, w^0_n}\vert_{[p_{q-1},p_{k_{i}}]}) \leq 10C_1^2\log(n)$ then $j=q-1$ proves the subclaim. If neither holds, then by $K$-bounded jumps of $(w^*_n)_n$ we get 
\begin{align*}
    K &\geq d(w^{o}_{q-1}, w^{o}_q)  \\
    &\geq D^{-1}\ell(\alpha_{o, w^0_n}\vert_{[p_{q-1},p_{q}]})-D-d(p_{q-1}, w^{o}_{q-1})-d(p_q, w^{o}_q) \\
    &\geq 20C_1^2\log(n)-D-2C_1^2\log(n),
\end{align*} 
a contradiction. This finishes the proof of \cref{subclaim:later_on_close}
\quad\hfill{$\diamondsuit$}
We are assuming that the event $\mathcal B_n$ holds, hence as $\vert j-k_i \vert \geq \vert k_{i+1}-k_i  \vert \geq C_3\log(n)$ we must have that $d_S(w^{o}_j, w^{o}_{k_i}) > \vert j-k_i \vert /C_3 \geq 100 C_1^2D^2\log(n)$. On the other hand, by Subclaim \ref{subclaim:later_on_close}, we also have that $d_S(w^{o}_j, w^{o}_{k_i}) \leq \ell (\gamma_j)+\ell(\gamma_{k_i})+ d_S(p_j, p_{k_i}) \leq 2C_1\log(n) + 10C_1^2D^2 \log(n)+D$, by the fact that $\alpha$ is a $D$-quasi-geodesic. A contradiction, thus proving Claim \ref{claim:projections_increase}. 
\hfill$\blacksquare$

Hence, the lengths $(\ell(\alpha_{o, w^0_n}\vert_{[o,p_{k_{i}}]}))_{k_i}$ are increasing. We further note that if $\mathcal B_n$ holds, then $d_S(w^{o}_{k_i}, w^{o}_{k_{i+1}}) \geq 100D^2C_1^2\log(n)$ and so $d_S(p_{k_i}, p_{k_{i+1}}) \geq 98C_1^2 \log(n)$. 

Thus under the assumption that events $\mathcal A_n$ and $\mathcal B_n$ hold we get, by \cref{claim:projections_increase} and $K$-bounded jumps of the Markov chain:
\begin{align*}
    \begin{split}
        Kn &\geq d_S(o,w^{o}_n) \geq \frac{1}{D}\ell(\alpha)-1   \\
        &\geq \frac{1}{D}\sum_{0\leq i\leq t} (\ell(\alpha_{o, w^0_n}\vert_{[p_{k_i},p_{k_{i+1}}]}))-1 \\
        &\geq \frac{t}{D}\Big(\frac{98C_1^2\log(n)}{D}-D \Big)-1.
    \end{split}
\end{align*}
Hence, $t \leq \frac{D^2(Kn+1)}{98C_1^2\log(n)-D}$. Further, for all $k_i$, consider loops $P_{k_{i}}$ to be defined by the concatenation of $\gamma^{-1}_{k_i}$, $[w_{k_i}, w_{k_{i+1}}], \gamma_{k_{i+1}}$, and $\alpha_{o, w^0_n}\vert_{[p_{k_i}, p_{k_{i+1}}]}$. The length of each loop $P_{k_i}$ is at most $300KC_1^2C_3D^2\log(n)$ and hence the area of each $P_{k_i}$ is at most $\delta_G( 300KC_1C_3D^2\log(n))$. By filling all of them, we get a filling for the null-homotopic word $W_{o, w_1^o, \ldots, w^o_n}$ (defined as in \cref{def:fill}).

Hence, in the case that the events $\mathcal A_n$ and $\mathcal B_n$ holds, we get that \begin{align*}
    \begin{split}
        \e \Big[ \Fill_{\calp, \alpha}(o, \dots, w^{o}_n) \Big\vert \mathcal A_n \cap \mathcal B_n \Big] &\leq \sum_{0 \leq i \leq t} \area{\calp}{P_{k_i}} \\
        &\leq t \delta_G(300KC_1^2C_3D^2\log(n)) \\
        &\leq \frac{D^2(Kn+1)}{98C_1^2\log(n)-D}\delta_G(300KC_1^2C_3D^2\log(n)).
    \end{split}
\end{align*}

And thus, by the law of total expectation

\begin{align*}
    \begin{split}
        \e \Big[ &\Fill_{\calp, \alpha}(o, \dots, w^{o}_n)  \Big] \\ &= \e \Big[ \Fill_{\calp, \alpha}(o, \dots, w^{o}_n) \Big\vert \mathcal A_n \cap \mathcal B_n \Big]\p[\mathcal A_n \cap \mathcal B_n]+\e \Big[ \Fill_{\calp, \alpha}(o, \dots, w^{o}_n)\Big\vert \mathcal A_n^C \cup \mathcal B_n^C \Big]\p[\mathcal A_n^C \cup \mathcal B_n^C] \\
        &\leq \frac{D^2(Kn+1)}{98C_1^2\log(n)-D}\delta_G(300KC_1^2C_3D^2\log(n)) + ((K+D)n+D)^d(C_1n^{1-C_1}+n^{-d-1}) \\
        &\leq C\Big(\frac{n}{\log(n)}\delta_G(\log(n))\Big)+Cn+C
    \end{split}
\end{align*}

for some $C>0$ big enough, where we use the fact that for two events $\mathcal A_n,\mathcal B_n$ we have $\p \Big[ \mathcal A_n \cap \mathcal B_n \Big] \geq 1-\p[\mathcal A_n^C]-\p[\mathcal B_n^C]$ and use Theorem \ref{thm:small_devns}, Lemma \ref{lem:subwalks_linear_progress} and the choice of $C_1$ to conclude.

This bound is independent of $o$ and hence taking the supremum over all $o \in G$ concludes the proof of Theorem \ref{thm:ah_rdehn}, by the usual order of functions defined above. Therefore,
	\begin{equation}\label{eq:general-bound-random}
		R\delta_{(G,(w_n^*), \alpha)}(n) \preccurlyeq \frac{n}{\log n} \delta_G(\log n)
	\end{equation}

By hypothesis, we have that $\delta_G(n) \preccurlyeq n^d$. Since $(\log n)^\alpha \preccurlyeq n$ for $\alpha >0$, it follows from \eqref{eq:general-bound-random} that
	\[
		R\delta_{(G,(w_n^*), \alpha)}(n) \preccurlyeq n (\log n)^{d-1} \preccurlyeq n^2
	\]
In particular, if $G$ is hyperbolic, that is if $d=1$, then we have that $R\delta_{(G,(w_n^*), \alpha)}(n) \preccurlyeq n$. Whereas, if $G$ is not hyperbolic and $\delta_G(n) \asymp n^2$, then 
	\[
		R\delta_{(G,(w_n^*), \alpha)}(n) \preccurlyeq n \log n \prec \delta_G(n);
	\] 
if $G$ is not hyperbolic and $n^2 \prec \delta_G(n) \preccurlyeq n^d$, then $R\delta_{(G,(w_n^*), \alpha)}(n) \prec \delta_G$. 
\end{proof}

As a consequence of \cref{thm:ah_rdehn} we obtain the main results stated in the introduction, which we state again for the reader's convenience.

\begin{corollary}[\cref{thmintro:ah_rdehn}]
Let $G$ be a finitely presented acylindrically hyperbolic group with Dehn function $\delta_G$ at most a polynomial of degree $d \geq 1$. Let $(Z_n)_n$ be a simple random walk on $G$. For every quasi-geodesic combing $\alpha$ of $G$, we have that
	\[
		R\delta_{(G,(Z_n), \alpha)}(n) \preccurlyeq n^2.
	\]
Furthermore, if $G$ is hyperbolic, then $R\delta_{(G,(Z_n), \alpha)}(n) \preccurlyeq n$. Else, if $d=2$ then $R\delta_{(G,(Z_n), \alpha)}(n) \preccurlyeq n \log n \prec \delta_G (n)$, and if $d>2$, then $R\delta_{(G,(Z_n), \alpha)} \prec \delta_G$.
\end{corollary}

\begin{proof}
    This follows immediately from Theorem \ref{thm:small_devns} and Theorem \ref{thm:ah_rdehn} by noting that the random walks considered are tame Markov chains by Lemma \ref{lem:random_walks_are_tame}.
\end{proof}

\begin{corollary}[\cref{thmintro:tame_MC}]
Let $G$ be one of the following:
\begin{itemize}
    \item A non-elementary hyperbolic group;
    \item a non-elementary relatively hyperbolic group; 
    \item an acylindrically hyperbolic $3$-manifold group;
    \item the fundamental group of a graph of groups with nilpotent undistorted edge groups and at most polynomial Dehn function
\end{itemize}
Let $(w_n^\ast)$ be a tame Markov chain on $G$. Then, for any quasi-geodesic combing $\alpha$ of $G$ we have that: 
\[
R\delta_{(G,(w^*_n), \alpha)}(n) \preccurlyeq  \frac{n}{\log(n)}\delta_G(\log(n)) \preccurlyeq n^2.
\]
In particular, if $G$ is not hyperbolic, then $R\delta_{(G,(w^*_n), \alpha)}(n) \preccurlyeq n \log n \prec \delta_G(n)$. 
\end{corollary}
\begin{proof} 
    This follows from Theorem \ref{thm:small_devns} as all these groups satisfy either \cite[Assumption 1 or Assumption 2]{GS21} and hence we can apply Theorem \ref{thm:ah_rdehn} by noting that all acylindrically hyperbolic $3$-manifold groups have at most quadratic Dehn function (by the fact that they have an automatic structure \cite[Theorem 12.4.7]{epstein} as they do not contain $Sol$ or $Nil$ components in their geometric decomposition), to conclude. 
\end{proof}

\begin{corollary}[\cref{thmintro:cor}]
    Let $G$ be one of the following:
    \begin{itemize}
        \item the mapping class group of a finite type surface;
        \item an extra-large Artin group; 
        \item or more generally a well-behaved hierarchically hyperbolic group (in the sense of \cite{InducedQIs}).
    \end{itemize}
    Let $(w_n^\ast)$ be a tame quasi-homogeneous Markov chain on $G$. Then, for any quasi-geodesic combing $\alpha$ of $G$ we have that:
    \[
    	R\delta_{(G,(w^*_n), \alpha)}(n) \preccurlyeq  \frac{n}{\log(n)} \delta_G(\log(n)) \preccurlyeq n^2.
    \]
    In particular, if $G$ is not hyperbolic, then $R\delta_{(G,(w^*_n), \alpha)}(n) \preccurlyeq n \log n \prec \delta_G(n)$.
\end{corollary}

\begin{proof}
We know that tame quasi-homogenenous Markov chains satisfy the conclusion of Theorem \ref{thm:small_devns}. Hence, applying Theorem \ref{thm:ah_rdehn} and using the fact that hierarchically hyperbolic groups have at most quadratic Dehn function \cite[Corollary 7.5]{BHS1} concludes the proof.
\end{proof}

\bibliography{main}
\bibliographystyle{alpha}

\end{document}